\documentclass[a4paper]{amsart}
\usepackage[english]{babel}
\usepackage[utf8x]{inputenc}
\usepackage[T1]{fontenc}
\usepackage{amsthm}
\usepackage{amssymb}
\usepackage{amsmath}
\theoremstyle{plain}
\usepackage{chngcntr}
\usepackage{relsize}
\usepackage{pdfpages}

\newtheorem{Lemma}{Lemma}
\newtheorem{Theorem}[Lemma]{Theorem}
\newtheorem{Proposition}[Lemma]{Proposition}
\newtheorem{Corollary}[Lemma]{Corollary}

\usepackage[a4paper,top=3cm,bottom=2cm,left=3cm,right=3cm,marginparwidth=1.75cm]{geometry}

\usepackage{float}
\usepackage{nccmath}
\usepackage{mathtools}
\usepackage{amsfonts}
\usepackage{amsmath}
\usepackage{graphicx}
\usepackage[colorinlistoftodos]{todonotes}
\usepackage[colorlinks=true, allcolors=blue]{hyperref}

\title{Small solutions of generic ternary quadratic congruences}



\author{Stephan Baier}
\address{Stephan Baier,
Ramakrishna Mission Vivekananda Educational and Research Institute, Department of Mathematics, G. T. Road, PO Belur Math, Howrah, West Bengal 711202, India}
\email{stephanbaier2017@gmail.com}
\author{Aishik Chattopadhyay}
\address{Aishik Chattopadhyay,
Ramakrishna Mission Vivekananda Educational and Research Institute, Department of Mathematics, G. T. Road, PO Belur Math, Howrah, West Bengal 711202, India}
\email{aishik.ch@gmail.com}
\begin{document}
\maketitle

\begin{abstract} We consider small solutions of quadratic congruences of the form $x_1^2+\alpha_2x_2^2+\alpha_3x_3^2\equiv 0 \bmod{q}$, where $q=p^m$ is an odd prime power. Here, $\alpha_2$ is arbitrary but fixed and $\alpha_3$ is variable, and we assume that $(\alpha_2\alpha_3,q)=1$. We show that for all $\alpha_3$ modulo $q$ which are coprime to $q$ except for a small number of $\alpha_3$'s, an asymptotic formula for the number of solutions $(x_1,x_2,x_3)$ to the congruence $x_1^2+\alpha_2x_2^2+\alpha_3x_3^2\equiv 0 \bmod{q}$ with $\max\{|x_1|,|x_2|,|x_3|\}\le N$ holds if $N\ge q^{11/24+\varepsilon}$  as $q$ tends to infinity over the set of all odd prime powers. It is of significance that we break the barrier 1/2 in the above exponent. If $q$ is restricted to powers $p^m$ of a {\it fixed} prime $p$ and $m$ tends to infinity, we obtain a slight improvement of this result using the theory of $p$-adic exponent pairs, as developed by Mili\'cevi\'c, replacing the exponent $11/24$ above by $11/25$. Under the Lindel\"of hypothesis for Dirichlet $L$-functions, we are able to replace the exponent $11/24$ above by $1/3$. 
\end{abstract}

\tableofcontents

\section{Introduction}
The Oppenheim conjecture, proved in full generality by Margulis \cite{Mar}, asserts that if $Q(x_1,...,x_n)\in \mathbb{R}[x_1,...,x_n]$ is an indefinite, non-degenerate quadratic form in $n\ge 3$ variables which is not a multiple of a quadratic form with integer coefficients, then $Q(\mathbb{Z}^n)$ is dense in $\mathbb{R}$. For $n\ge 5$, this problem is in reach of the circle method, producing quantitative results about approximations of real numbers by values of $Q(x_1,...,x_n)$, where $(x_1,...,x_n)\in \mathbb{Z}^n$ (for an early result in this direction on diagonal forms in five variables, see \cite{DH}). Margulis \cite{Mar} proved the above conjecture using ergodic theoretic methods. This approach yields only a qualitative result, though. Bourgain \cite{JB} obtained a quantitative result for {\it generic} diagonal quadratic forms  in $n=3$ variables, relating this problem to Dirichlet polynomials. In particular, with regard to approximations of $0$ by $Q(x_1,x_2,x_3)$, he established the following.

\begin{Theorem}[Bourgain] \label{Bour}
Fix positive real numbers $\varepsilon$ and $\alpha_2$. Then, as $N\rightarrow \infty$, for all $\alpha_3\in [1/2,1]$, except for a set of Lebesgue measure $o(1)$, we have 
$$
\min\limits_{\substack{(x_1,x_2,x_3)\in \mathbb{Z}^3\\ 0<\max\{|x_1|,|x_2|,|x_3|\}\le N}} |x_1^2+\alpha_2x_2^2-\alpha_3x_3^2|<N^{-2/5+\varepsilon}.
$$
Under the Lindel\"of hypothesis for the Riemann zeta function, the exponent $2/5$ can be replaced by $1$.  
\end{Theorem}

In fact, Bourgain's method yields a lower bound of the expected order of magnitude for the number of triples $(x_1,x_2,x_3)\in \mathbb{Z}^3$ satisfying $0<\max\{|x_1|,|x_2|,|x_3|\}\le N$ and $|x_1^2+\alpha_2x_2^2-\alpha_3x_3^2|<N^{-2/5+\varepsilon}$ for generic quadratic forms (in the sense that the Lebesgue measure of exceptional $\alpha_3$'s tends to $0$ as $N\rightarrow \infty$).
Here we consider a related question in the $p$-adic setting. Fix an odd prime $p$. For generic quadratic forms $x_1^2+\alpha_2x_2^2+\alpha_3x_3^2$ with integer coefficients coprime to $p$, we aim to establish a lower bound for the number of triples
$(x_1,x_2,x_3)\in \mathbb{Z}^3$ satisfying  $(x_3,p)=1$, $\max\{|x_1|,|x_2|,|x_3|\}\le N$ and  $|x_1^2+\alpha_2x_2^2+\alpha_3x_3^2|_p<N^{-\gamma}$,
as $N$ tends to infinity. 
Here, $|.|_p$ denotes the $p$-adic norm, and $\gamma$ is an as large as possible exponent. We are able to achieve the said lower bound for any $\gamma<25/11$ unconditionally and for any $\gamma<3$ under the Lindel\"of hypothesis for Dirichlet $L$-functions. Here it is of significance that $25/11$ and $3$ are greater than $2$, as we will explain below. Precisely, we prove the following asymptotic result, from which the said lower bound will follow.

\begin{Theorem} \label{mainresult} {\rm (i)} Fix $\varepsilon>0$. Then as $q$ tends to infinity over the set of all odd prime powers, given any $N\in \mathbb{R}$ and $\alpha_2\in \mathbb{Z}$ satisfying $q^{11/24+\varepsilon}\le 2N\le q^{7/12}$ and $(\alpha_2,q)=1$, the asymptotic formula
\begin{equation} \label{asymp}
\sum\limits_{\substack{|x_1|,|x_2|,|x_3|\le N\\ (x_3,q)=1\\ 
x_1^2+\alpha_2x_2^2+\alpha_3x_3^2\equiv 0 \bmod{q}}} 1=C_q\cdot \frac{N^3}{q}\cdot \left(1+o(1)\right)
\end{equation} 
holds for all
$$
\alpha_3\in \{s\in \mathbb{N} : 1\le s\le q, \ (s,q)=1\}
$$ 
with the exception of at most $o(\varphi(q))$ elements $\alpha_3$. Here,
\begin{equation}\label{Cpdef}
C_q:=8\left(1-\frac{1}{p}\right)\left(1-\frac{1}{p}\cdot \left(\frac{-\alpha_2}{p}\right)\right)
\end{equation}
if $q=p^m$ with $p$ an odd prime and $m\in \mathbb{N}$, where 
$\left(\frac{\cdot}{p}\right)$ denotes the Legendre symbol. \medskip\\
{\rm (ii)} Fix an odd prime $p$. Then the exponent $11/24$ in part {\rm (i)} can be replaced by $11/25$ if $q$ tends to infinity over all powers $p^m$ with $m\in \mathbb{N}$ of this fixed prime (i.e. $m\rightarrow \infty$).\medskip\\
{\rm (iii)} Under the Lindel\"of hypothesis for Dirichlet $L$-functions, the exponent $11/24$ in part {\rm (i)} can be replaced by $1/3$.
\end{Theorem} 

We point out that the exceptional set of $\alpha_3$'s above depends on both $q$ and $\alpha_2$. In this article, we don't investigate its properties, though. 

The condition $2N\le q^{7/12}$ in Theorem \ref{mainresult} comes from an application of Proposition \ref{HeathBrown} due to Heath-Brown below, where the parameter $R$ never exceeds $q^{7/12}$. Extending Heath-Brown's result to larger $R$'s would allow us to relax this condition. With some efforts, this is certainly possible, but since we are mainly interested in small $N$'s, we here abstain from carrying this out.
   
We also remark that in part (ii) above, rather than fixing a prime $p$ and letting $m$ tend to infinity, it suffices to assume that $q=p^m$ with $m\ge m(p)$, where $m(p)$ is an explicitly computable constant depending on $p$. This is so because the $p$-dependence in Proposition \ref{Mil} below is an explicitly computable power of $p$. However, for the corollary below, part (ii) above is sufficient since we consider a fixed prime in this setting.  

We recall that for $s\in \mathbb{Z}$, $s\equiv 0\bmod{p^m}$ is equivalent to $|s|_p\le p^{-m}$. Hence, as announced above, we deduce the following lower bound from parts (ii) and (iii) of Theorem \ref{mainresult}.

\begin{Corollary} \label{reform} Fix an odd prime $p$ and a positive real number $\gamma<25/11$. For any $N\in \mathbb{N}$ fix an arbitrary integer $\alpha_2=\alpha_2(N)$ which is coprime to $p$. Then, as $N$ runs over the natural numbers, we have a lower bound of the form
\begin{equation*}
\sum\limits_{\substack{|x_1|,|x_2|,|x_3|\le N\\ (x_3,p)=1\\ 
|x_1^2+\alpha_2x_2^2+\alpha_3x_3^2|_p\le N^{-\gamma}}} 1\gg_p N^{3-\gamma} 
\end{equation*}
for all $\alpha_3\in \mathbb{Z}$ coprime to $p$, except for a set of integers $\alpha_3$ whose density in $\mathbb{Z}$ tends to $0$ as $N$ tends to infinity. 
Under the Lindel\"of hypothesis for Dirichlet $L$-functions, the above condition on $\gamma$ can be replaced by $\gamma<3$.  
\end{Corollary}

This should be compared to Theorem \ref{Bour}. It is important to point out that the exponents $11/24=0.458\overline{3}$, $11/25=0.44$ and $1/3=0.\overline{3}$ in Theorem \ref{mainresult} are less than $1/2$ (and hence the corresponding reciprocal exponents in Corollary \ref{reform} are greater than 2). The following simple example shows that in general, it is not possible to obtain an exponent smaller than $1/2$ if we fix both the coefficients $\alpha_2$ and $\alpha_3$: If we take $\alpha_2=1=\alpha_3$, then the congruence $x_1^2+\alpha_2x_2^2+\alpha_3x_3^2\equiv 0\bmod{q}$ has exactly one solution $(x_1,x_2,x_3)\in \mathbb{Z}^3$ satisfying $\max\{|x_1|,|x_2|,|x_3|\}< \sqrt{q/3}$, namely the trivial solution $(0,0,0)$. Hence, in this case, the asymptotic \eqref{asymp} breaks down if $N<\sqrt{q/3}$. Consequently, for suitable $\alpha_2$, there exist exceptional $\alpha_3$'s in Theorem \ref{mainresult} such that the said asymptotic does not hold if $N$ is significantly smaller than $q^{1/2}$. On the other hand, an asymptotic formula similar to that in \eqref{asymp} was established in \cite[Theorem1]{snu} for $N\ge q^{1/2+\varepsilon}$ if both $\alpha_2$ and $\alpha_3$ are fixed and $q$ tends to infinity over the powers of an odd prime $p$. (In \cite{snu}, the coprimality condition on the $x_i$'s was slightly different, namely $(x_1x_2x_3,q)=1$ in place of $(x_3,q)=1$. This changes the constant $C_q$ but is otherwise not of much significance.) Summarizing the above considerations, Theorem \ref{mainresult} breaks the $1/2$-barrier for almost all but not all forms. In other words, we obtain an exponent smaller than $1/2$ for {\it generic} forms. Naturally, this raises the question whether there is a simple characterization of the exceptional forms, which may be addressed in future work.

We also mention that if $Q(x_1,x_2,x_3)$ is a diagonal form with coefficients which are allowed to vary with the modulus $q=p^m$, then an asymptotic formula similar to the above was established in \cite[Theorem 2]{snu} under the condition $N\ge q^{11/18+\varepsilon}$. This should be compared to a result of Heath-Brown who established in \cite[Theorem 1]{HB} that there exists a non-trivial solution of the congruence $Q(x_1,x_2,x_3)\equiv 0 \bmod{q}$ with $\max\{|x_1|,x_2|,|x_2|\}\le q^{5/8+\varepsilon}$ for any odd {\it squarefree} modulus $q$ and any quadratic form $Q$ with $(\det(Q),q)=1$.  \\ \\
{\bf Acknowledgements.} The authors would like to thank the Ramakrishna Mission Vivekananda Educational and Research Insititute for an excellent work environment. The research of the second-named author was supported by a CSIR Ph.D fellowship under file number 09/0934(13170)/2022-EMR-I. 

\section{Notations and preliminaries} 
Throughout this paper, we follow the usual convention that $\varepsilon$ is an arbitrarily small positive number. We also assume that $\varepsilon$ is small enough so that our arguments go through.  Multiplicative inverses of integers $a$ are denoted  by $\overline{a}$, where the modulus will be clear from the context. So if $c\in \mathbb{N}$ is the modulus and $a\in \mathbb{Z}$ is coprime to $c$, then $\overline{a}$ is an integer such that $a\overline{a}\equiv 1\bmod{c}$. 
  
We will use the following character sum estimates due to Burgess, Heath-Brown and Mili\'cevi\'c in this paper.
 
\begin{Proposition}[Burgess] \label{Burgess}
Fix $\varepsilon>0$. Let $M>0$, $N\in \mathbb{N}$ and $\chi$ be a primitive Dirichlet character of conductor $q>1$. Then
$$
\sum\limits_{M<n\le M+N} \chi(n) \ll_{\varepsilon,r} N^{1-1/r}q^{(r+1)/(4r^2)+\varepsilon}
$$ 
for $r=2,3$, and for any $r\in \mathbb{N}$ if $q$ is cubefree. 
\end{Proposition}

\begin{proof}
This is \cite[Theorem 12.6]{IwKo}. 
\end{proof}

\begin{Proposition}[Heath-Brown] \label{HeathBrown}
Let $\varepsilon > 0$ and an integer $r\ge 3$ be given, and suppose that $C\subset \mathbb{R}^2$ is a convex set contained in a disc $\{{\bf x} \in \mathbb{R}^2:
 ||{\bf x}-{\bf x}_0||_2 \le R\}$, $||.||_2$ denoting the Euclidean norm. Let $q \ge 2$ be odd and squarefree, and let $\chi$ be a primitive character to modulus $q$. Then if $Q(x,y)$ is a binary integral quadratic form with $(\det(Q), q) = 1$, we have
$$
\sum\limits_{(x,y)\in C}  \chi(Q(x, y)) \ll_{\varepsilon,r} R^{2-1/r}q^{(r+2)/(4r^2)+\varepsilon}
$$
for $q^{1/4+1/(2r)} \le R \le q^{5/12+1/(2r)}$.
\end{Proposition}

\begin{proof}
This is \cite[Theorem 3]{HB}. 
\end{proof}

As pointed out by Heath-Brown, this result was inspired by an earlier result of Chang \cite{Cha} considering character sums of binary quadratic forms over boxes. 

\begin{Proposition}[Mili\'civi\'c] \label{Mil}
Fix a positive real number $\varepsilon$ and an odd prime $p$. Let $N\in \mathbb{N}$ and $q:=p^m$, where $m\in \mathbb{N}$. Let $\chi$ be a primitive character modulo $q$. Then, for any exponent pair $(k,l)$, we have 
$$
\sum\limits_{0<n\le N} \chi(n)\ll_{\varepsilon,p,k,l} N^{l-k}q^{k+\varepsilon}.
$$
Here, we term $(k,l)$ an exponent pair if it belongs to the set which is obtained by iteratively applying the following linear fractional transformations $A,B:\mathbb{R}^2\rightarrow \mathbb{R}^2$, starting from the trivial exponent pair $(0,1)$: 
$$
A(k,l):=\left(\frac{k}{2k+2},\frac{k+l+1}{2k+2}\right)
$$ 
and 
$$
B(k,l):=\left(l-\frac{1}{2},k+\frac{1}{2}\right). 
$$
\end{Proposition}

\begin{proof} If $m$ is large enough, then similarly as in \cite[equation (58)]{Mil}, we may write
\begin{equation} \label{chisum}
\sum\limits_{0<n\le N} \chi(n)=\sum\limits_{\substack{1\le c\le p^{\kappa}\\ (c,p)=1}} \chi(c) \cdot \sum\limits_{0\le r\le (N-c)/p^{\kappa}} e\left(\frac{f_c(r)}{q}\right)
\end{equation}
with
$$
f_c(r)= a_0\log_p\left(1+p^{\kappa}\overline{c}r\right) 
$$
for suitable $\kappa\in \mathbb{N}$ and $a_0\in \mathbb{Z}_p^{\times}$ which are independent of $m$. Furthermore, using \cite[estimate (59)]{Mil}, we have 
\begin{equation} \label{fsum}
\sum\limits_{0\le r\le t} e\left(\frac{f_c(r)}{q}\right)\ll_{\varepsilon,p,k,l} \left(\frac{q}{t}\right)^kt^lq^{\varepsilon}
\end{equation}
for any exponent pair $(k,l)$. Combining \eqref{chisum} and \eqref{fsum}, the claimed estimate follows.  
\end{proof}

Under the Lindel\"of Hypothesis for Dirichlet $L$-functions, we have the following sharper estimate.

\begin{Proposition}\label{Linde}
Fix $\varepsilon>0$. Let $N\in \mathbb{N}$ and $\chi$ be a primitive Dirichlet character of conductor $q>1$. Then
$$
\sum\limits_{0<n\le N} \chi(n) \ll_{\varepsilon} N^{1/2}q^{\varepsilon},
$$ 
provided that $L(\sigma+it,\chi)\ll_{\varepsilon} (1+|t|)^{\varepsilon}q^{\varepsilon}$ for all $\sigma\ge 1/2$ and $t\in \mathbb{R}$. 
\end{Proposition}

\begin{proof} Perron's formula implies that
$$\sum_{0<n\leq N}\chi(n)=\frac{1}{2\pi i}\int\limits_{c-iT}^{c+iT}L(\chi,s)\frac{N^s}{s}ds +O\left(\frac{N\log N}{T}+1\right)$$ 
for $c=1+1/\log{N}$ and any $T\ge 1$. By Cauchy's integral theorem, we may write the integral on the right-hand side as 
\begin{align*}
 \int\limits_{c-iT}^{c+iT}L(\chi,s)\frac{N^s}{s}ds = \left(\int\limits_{c-iT}^{1/2-iT}+\int\limits_{1/2-iT}^{1/2+iT}+\int\limits_{1/2+iT}^{c+iT}\right) L(\chi,s)\frac{N^s}{s}ds=: I_1+I_2+I_3.
\end{align*}
We also have
$$
L(s,\chi)=\prod_{\substack{p|q\\p\nmid q_1}}\left(1-\chi(p)p^{-s}\right)L(s,\chi_{1}),
$$
where $q_1$ is the conductor of $\chi_1$. From  
$$
\prod_{\substack{p|q\\p\nmid q_1}}\left(1-\chi(p)p^{-(\sigma+it)}\right)\ll q^{\varepsilon}
$$
and our assumption $L(\sigma+it,\chi_1)\ll (1+|t|)^{\varepsilon}q^{\varepsilon}$ if $\sigma\ge 1/2$ and $t\in \mathbb{R}$, it now follows that 
$$
 L(\chi,\sigma+it) \cdot \frac{N^{\sigma+it}}{\sigma+it}\ll \frac{(1+|t|)^{\varepsilon}q^{2\varepsilon}N^{\sigma}}{|t|}
$$
under these conditions. Hence, $I_1,I_3\ll (Tq)^{2\varepsilon} NT^{-1}$ and $I_2\ll (Tq)^{3\varepsilon}N^{1/2}$. This implies the claim upon choosing $T:=N$ and re-defining $\varepsilon$.  
\end{proof}

We will also use the following well-known results about quadratic Gauss sums. 

\begin{Proposition} \label{Gauss sums}
Let $p$ be an odd prime and $a,b\in \mathbb{Z}$, where $(a,p)=1$. Set  
\begin{equation} \label{Gaussdef}
G(a,b;p):=\sum\limits_{n=1}^p e\left(\frac{an^2+bn}{p}\right)
\end{equation}
and 
\begin{equation} \label{Gaussmuldef}
\tau_p:=\sum_{n=1}^{p}\left(\frac{n}{p}\right)e\left(\frac{n}{p}\right).
\end{equation}
Then 
\begin{equation} \label{Gaussev}
G(a,b;p)=e\left(-\frac{\overline{4a}b^2}{p}\right)\cdot \left(\frac{a}{p}\right)\cdot \tau_p,
\end{equation}
where 
\begin{equation} \label{tauev}
\tau_p=\epsilon_p\sqrt{p}=\begin{cases} 
      \sqrt{p} & \text{if } p\equiv 1 \bmod{4} \\
      i\sqrt{p} & \text{if } p\equiv 3 \bmod{4}.
\end{cases}
\end{equation}
Moreover, for all $n\in \mathbb{Z}$, we have the relation
\begin{equation} \label{relation}
\left(\frac{n}{p}\right)=\frac{1}{\tau_p}\sum_{k=1}^{p}\left(\frac{k}{p}\right)e\left(\frac{nk}{p}\right).
\end{equation} 
\end{Proposition}

\begin{proof} Quadratic completion and a change of variables $k=n+\overline{2a}b$ give 
$$
G(a,b;p)=\sum\limits_{n=1}^p e\left(\frac{a(n+\overline{2a}b)^2-\overline{4a}b^2}{p}\right)=e\left(\frac{-\overline{4a}b^2}{p}\right)\sum\limits_{k=1}^p e\left(\frac{ak^2}{p}\right).
$$
Furthermore, \cite[Theorem 1.1.5]{BEW} and a change of variables $m=ak$ imply 
$$
\sum\limits_{k=1}^p e\left(\frac{ak^2}{p}\right)= \sum\limits_{k=1}^p \left(\frac{k}{p}\right)e\left(\frac{ak}{p}\right)=\left(\frac{\overline{a}}{p}\right)\sum\limits_{m=1}^p \left(\frac{m}{p}\right)e\left(\frac{m}{p}\right)=\left(\frac{a}{p}\right)
\tau_p.
$$
From the above, the relations \eqref{Gaussev} and \eqref{relation} follow. The relation \eqref{tauev} can be found in \cite[Lemma 1.2.1]{BEW}. 
\end{proof}

\section{Splitting into main and error term}
Suppose that the conditions in Theorem \ref{mainresult} are satisfied, i.e., $q=p^m$ tends to infinity over the set of all odd prime powers, and $N\in \mathbb{R}$ and $\alpha_2\in \mathbb{Z}$ satisfy the conditions $q^{11/24+\varepsilon}\le 2N\le q^{7/12}$ and $(\alpha_2,q)=1$. Set 
$$
S(\alpha_3):=\sum\limits_{\substack{|x_1|,|x_2|,|x_3|\le N\\ (x_3,q)=1\\ 
x_1^2+\alpha_2x_2^2+\alpha_3x_3^2\equiv 0 \bmod{q}}} 1.
$$
Our basic approach is to detect the congruence condition 
$$
x_1^2+\alpha_2x_2^2+\alpha_3x_3^2\equiv 0 \bmod{q}
$$
via orthogonality relations for Dirichlet characters. Recalling the condition $(\alpha_3x_3,q)=1$, we have 
$$
\frac{1}{\varphi(q)}
\sum\limits_{\chi \bmod q} \chi\left(x_1^2+\alpha_2x_2^2\right)\overline{\chi}\left(-\alpha_3 x_3^2\right) = \begin{cases} 1 & \mbox{ if } x_1^2+\alpha_2x_2^2+\alpha_3x_3^2\equiv 0 \bmod{q}\\ 0 & \mbox{ if }x_1^2+\alpha_2x_2^2+\alpha_3x_3^2\not\equiv 0 \bmod{q}. \end{cases}
$$
It follows that 
$$
S(\alpha_3)= \frac{1}{\varphi(q)} \sum\limits_{\chi \bmod{q}}\
\sum_{|x_1|,|x_2|,|x_3|\le N}  \chi\left(x_1^2+\alpha_2x_2^2\right)\overline{\chi}\left(-\alpha_3 x_3^2\right).
$$
The main term contribution comes from the principal character $\chi_0 \bmod{q}$. Thus we may split the above into
\begin{equation} \label{splitting}
S(\alpha_3)=M+E(\alpha_3),
\end{equation}
where 
$$
M:=\frac{1}{\varphi(q)} 
\sum_{\substack{|x_1|,|x_2|,|x_3|\le N\\ \left(x_1^2+\alpha_2x_2^2,q\right)=1\\  (x_3,q)=1}} 1
$$
is the main term and 
\begin{equation} \label{Edef}
E(\alpha_3):=\frac{1}{\varphi(q)} \sum\limits_{\substack{\chi \bmod{q}\\ \chi\not=\chi_0}}\
\sum_{|x_1|,|x_2|,|x_3|\le N}  \chi\left(x_1^2+\alpha_2x_2^2\right)\overline{\chi}\left(-\alpha_3 x_3^2\right)
\end{equation}
is the error term. 

\section{Evaluation of the main term}
In this section, we approximate the main term. We have 
\begin{equation} \label{maintermsplit}
M=\frac{1}{\varphi(q)}\cdot KL,
\end{equation}
where 
$$
K:=\sum_{\substack{|x_1|,|x_2|\le N\\ \left(x_1^2+\alpha_2x_2^2,p\right)=1}} 1
$$
and 
$$
L:=\sum_{\substack{|x_3|\le N\\ (x_3,p)=1}} 1,
$$
where we recall that $q=p^m$. Clearly,
\begin{equation} \label{Lapprox}
L=\sum_{|x_3|\le N} 1- \sum_{\substack{|x_3|\le N\\ p|x_3}} 1=2\left(1-\frac{1}{p}\right)N+O(1)
\end{equation}
and 
$$
 K:=\sum\limits_{|x_1|,|x_2|\leq N} 1-\sum\limits_{\substack{|x_1|,|x_2|\leq N\\ p|\left(x_1^2+\alpha_2x_2^2\right)}}1 = 4N^2-\sum\limits_{\substack{|x_1|,|x_2|\leq N\\ p|
\left(x_1^2+\alpha_2x_2^2\right)}}1 +O(N). 
$$
Further, we split the second sum over $x_1,x_2$ into two parts, according to whether $p|x_1$ or $p\nmid x_1$,  as follows. We have
\begin{equation*}
\begin{split}
\sum\limits_{\substack{|x_1|,|x_2|\leq N\\ p|(x_1^2+\alpha_2x_2^2)}}1  = & 
\sum\limits_{\substack{|x_1|\leq N\\ p|x_1}}\sum\limits_{\substack{x_2\in \mathbb{Z}\\ |x_2|\leq N\\ p|x_2}} 1 +\sum\limits_{\substack{|x_1|\leq N\\ p\nmid x_1}}\sum\limits_{\substack{|x_2|\leq N\\ x_1^2+\alpha_2x_2^2\equiv 0 \bmod{p}}} 1\\
= &
\frac{4N^2}{p^2}+O\left(\frac{N}{p}+1\right)+\sum\limits_{\substack{|x_1|\leq N\\ p\nmid x_1}}\sum\limits_{\substack{|x_2|\leq N\\ x_2^2\equiv -\overline{\alpha_2}x_1^2\bmod{p}}} 1.
\end{split}
\end{equation*}
For fixed $x_1\in \mathbb{Z}$ coprime to $p$, we have 
$$
\sum\limits_{\substack{|x_2|\leq N\\ x_2^2\equiv -\overline{\alpha_2}x_1^2\bmod{p}}} 1 = \frac{2N}{p}\cdot \left(\left(\frac{-\alpha_2}{p}\right)+1\right)+O(1).
$$ 
It follows that 
$$
\sum\limits_{\substack{|x_1|\leq N\\ p\nmid x_1}}\sum\limits_{\substack{|x_2|\leq N\\ x_2^2\equiv \overline{\alpha_2}x_1^2\bmod{p}}} 1 = \left(1-\frac{1}{p}\right)\cdot \frac{4N^2}{p}\cdot \left(\left(\frac{-\alpha_2}{p}\right)+1\right)+O\left(N\right). 
$$
Combining the above approximations, we obtain
\begin{equation} \label{Kapprox}
K=4N^2\left(1-\frac{1}{p}\right)\left(1-\frac{1}{p}\cdot \left(\frac{-\alpha_2}{p}\right)\right)+O(N).
\end{equation}
Combining \eqref{maintermsplit}, \eqref{Lapprox} and \eqref{Kapprox}, we obtain
\begin{equation} \label{Mapproxi}
M=\frac{1}{\varphi(q)}\cdot 8N^3\left(1-\frac{1}{p}\right)^2\left(1-\frac{1}{p}\cdot \left(\frac{-\alpha_2}{p}\right)\right)+O\left(\frac{N^2}{\varphi(q)}\right)=C_q\cdot \frac{N^3}{q}+O\left(\frac{N^2}{q}\right),
\end{equation}
where $C_q$ is defined as in \eqref{Cpdef}. 

\section{Estimation of the variance - initial steps}
To derive Theorem \ref{mainresult}, we will estimate the variance 
\begin{equation} \label{Vdefi}
V:=\sum\limits_{\substack{\alpha_3=1\\ (\alpha_3,q)=1}}^q \left|S(\alpha_3)-M\right|^2=\sum\limits_{\substack{\alpha_3=1\\ (\alpha_3,q)=1}}^q \left|E(\alpha_3)\right|^2.
\end{equation}
Our goal is to beat the estimate $O\left(N^6q^{-1}\right)$ in order to deduce that for almost all $\alpha_3 \bmod{q}$, the size of the error term $E(\alpha_3)$ is smaller than that of the main term $M$. Plugging in the right-hand side of \eqref{Edef} for $E(\alpha_3)$ and using orthogonality relations for Dirichlet characters, we have 
\begin{equation*}
\begin{split}
V=&\frac{1}{\varphi(q)^2} \sum\limits_{\alpha_3=1}^q \bigg|\sum\limits_{\substack{\chi \bmod{q}\\ \chi\neq\chi_0}}\overline{\chi}(-\alpha_3)\sum\limits_{|x_1|,|x_2|\leq N} \chi\left(x_1^2+\alpha_2 x_2^2\right)\sum\limits_{|x_3|\leq N}\overline{\chi}^2(x_3)\bigg|^2\\
=&\frac{1}{\varphi(q)^2}\sum\limits_{\substack{\chi_1,\chi_2\bmod{q}\\ \chi_1,\chi_2\neq \chi_0}}\ \sum\limits_{\alpha_3=1}^q \overline{\chi_1}\chi_2(-\alpha_3)\sum\limits_{|x_1|,|x_2|\leq N}\chi_1\left(x_1^2+\alpha_2x_2^2\right)\sum\limits_{|y_1|,|y_2|\leq N} \overline{\chi_2}\left(y_1^2+\alpha_2y_2^2\right)\times\\ & \sum\limits_{|x_3|\leq N} \overline{\chi_1}^2(x_3)  \sum\limits_{|y_3|\leq N} \chi_2^2(y_3)\\
=&\frac{1}{\varphi(q)}\sum\limits_{\substack{\chi \bmod{q}\\ \chi\neq\chi_0}}\bigg|\sum\limits_{|x_1|,|x_2|\leq N} \chi\left(x_1^2+\alpha_2x_2^2\right)\sum_{|x_3|\leq N}\overline{\chi}^2(x_3)\bigg|^2.
\end{split}
\end{equation*}
Next, we separate the summation into two parts: the contributions of characters $\chi$ with $\chi^2=\chi_0$ and $\chi^2\not=\chi_0$, respectively. We note that there is only one non-principal character $\chi$ modulo $q=p^m$ such that $\chi^2=\chi_0$, namely the Legendre symbol
$$
\chi(x)=\left(\frac{x}{p}\right).
$$ 
This is a consequence of the isomorphy of the character group modulo $q$ to $(\mathbb{Z}/q\mathbb{Z})^{\ast}$ and Hensel's lemma. Hence, we obtain
\begin{equation} \label{Vsplit}
V=V_1+V_2,
\end{equation}
where 
\begin{equation}\label{s1}
    V_1:=\frac{1}{\varphi(q)} \cdot  \bigg|\sum_{|x_1|,|x_2|\leq N} \left(\frac{x_1^2+\alpha_2x_2^2}{p}\right)\bigg|^2\cdot \bigg|\sum\limits_{\substack{|x_3|\le N\\ (x_3,p)=1}} 1\bigg|^2 
\end{equation}
and 
\begin{equation}\label{s2}
    V_2=\frac{1}{\varphi(q)} \sum\limits_{\substack{\chi\bmod{q}\\ \chi^2\neq \chi_0}}\bigg|\sum_{|x_1|,|x_2|\leq N} \chi\left(x_1^2+\alpha_2x_2^2\right) \bigg|^2 \cdot \bigg| \sum_{|x_3|\leq N} \overline{\chi}^2(x_3)\bigg|^2.
\end{equation}

\section{Estimation of $V_1$}
In this section, we estimate $V_1$, the contribution of the Legendre symbol to the variance $V$. We handle two cases in different ways.

\subsection{Prime moduli}
If $m=1$ and hence $q=p$ is a prime modulus, we use Proposition \ref{HeathBrown}
with $C:=\{(x_1,x_2)\in \mathbb{R}^2 : \max\{|x_1|,|x_2|\}\le N\}$, $x_0=0$, $R:=2N$ and $q=p$ to obtain 
\begin{equation} \label{V1first}
V_1\ll_{\varepsilon,r} N^{6-2/r}q^{(r+2)/(2r^2)-1+\varepsilon},
\end{equation}
provided that $r\ge 3$ is an integer such that 
\begin{equation} \label{condi}
q^{1/4+1/(2r)} \le 2N \le q^{5/12+1/(2r)}.
\end{equation} 
Recall that we aim to break the estimate $O(N^6q^{-1})$. For any $\Delta\in (0,1)$, the right-hand side of \eqref{V1first} is $O\left(\Delta N^6q^{-1}\right)$ if
$$
N\ge\Delta^{-r/2}q^{1/4+1/(2r)+\varepsilon}. 
$$
Combining this with \eqref{condi}, we require that 
\begin{equation*} 
\Delta^{-\frac{r}{2}}q^{1/4+1/(2r)+\varepsilon}\le N\le \frac{1}{2}\cdot q^{5/12+1/(2r)}.
\end{equation*}
We will take $\Delta:=q^{-2\varepsilon}$. It is then easily seen that the $N$-ranges above overlap for $r=3,4,5,6$, and we thus get  
\begin{equation} \label{condi2}
V_1=O_{\varepsilon}\left(\Delta N^6q^{-1}\right) \mbox{ if }\Delta^{-3}q^{1/3+\varepsilon}\le N\le \frac{1}{2}\cdot q^{7/12},
\end{equation}
where the lower bound for $N$ comes from the case $r=6$ and the upper bound from $r=3$. 

\subsection{Prime power moduli}
If $m\ge 2$ and hence $q=p^m$ is a prime power modulus, we argue by completing the character sum over $x_1$ and $x_2$ as follows. We begin by writing
\begin{equation} \label{write}
\begin{split}
& \sum_{|x_1|,|x_2|\leq N}\left(\frac{x_1^2+\alpha_2x_2^2}{p}\right)\\
=& \frac{1}{p^2}\sum_{a_1,a_2\bmod{p}} \left(\frac{a_1^2+\alpha_2a_2^2}{p}\right)\sum_{h_1,h_2\bmod p}\ \sum_{|x_1|,|x_2|\leq N} e\left(\frac{h_1(a_1-x_1)}{p}\right)e\left(\frac{h_2(a_2-x_2)}{p}\right)\\
= & \frac{1}{p^2}\sum_{h_1,h_2 \bmod p}\ \sum_{|x_1|\leq N} e\left(-\frac{h_1x_1}{p}\right)\sum_{|x_2|\leq N} e\left(-\frac{h_2x_2}{p}\right)\sum_{a_1,a_2 \bmod{p}}\left(\frac{a_1^2+\alpha_2a_2^2}{p}\right)\cdot e\left(\frac{h_1a_1+h_2a_2}{p}\right).
\end{split}
\end{equation}
Now using the relation \eqref{relation}, we deduce that 
  \begin{equation*}
\begin{split}
   \sum_{a_1,a_2 \bmod{p}}\left(\frac{a_1^2+\alpha_2a_2^2}{p}\right)\cdot e\left(\frac{h_1a_1+h_2a_2}{p}\right)
  =&\frac{1}{\tau_p}\sum\limits_{k=1}^{p}\left(\frac{k}{p}\right)\sum_{a_1,a_2\bmod p}e\left(\frac{ka_1^2+h_1a_1+k\alpha_2a_2^2+h_2a_2}{p}\right)\\
  =&\frac{1}{\tau_p}\sum\limits_{k=1}^{p}\left(\frac{k}{p}\right)G(k,h_1;p)G(k\alpha_2,h_2;p),
\end{split}
\end{equation*}
where $G(a,b;p)$ is the quadratic Gauss sum, defined in \eqref{Gaussdef}. Using its evaluation in \eqref{Gaussev}, it follows that 
\begin{equation*}
\begin{split}
  \sum_{a_1,a_2 \bmod{p}} \left(\frac{a_1^2+\alpha_2a_2^2}{p}\right)\cdot e\left(\frac{h_1a_1+h_2a_2}{p}\right)=& \left(\frac{\alpha_2}{p}\right)\cdot \tau_p\cdot \sum_{k=1}^{p-1} \left(\frac{k}{p}\right)e\left(-\frac{\overline{k}(\overline{4}h_1^2+\overline{4\alpha_2}h_2^2)}{p}\right)\\
= &\left(\frac{\alpha_2}{p}\right)\cdot \tau_p\cdot \sum_{k=1}^{p}\left(\frac{k}{p}\right)e\left(-\frac{k(\overline{4}h_1^2+\overline{4\alpha_2}h_2^2)}{p}\right).
\end{split}
\end{equation*}
Again using \eqref{relation} and \eqref{tauev}, we conclude that
\begin{equation*}
\begin{split}
\sum_{a_1,a_2 \bmod{p}} \left(\frac{a_1^2+\alpha_2a_2^2}{p}\right)\cdot e\left(\frac{h_1a_1+h_2a_2}{p}\right)
=  & \left(\frac{-(\alpha_2h_1^2+h_2^2)}{p}\right)\cdot \tau_p^2\\
& \begin{cases} 
      =0 & \text{if } (h_1,h_2)\equiv (0,0) \bmod{p}, \\
      \ll p & \text{otherwise.}
   \end{cases}
\end{split}
\end{equation*}
Plugging this into the last line of \eqref{write} and using triangle inequality, we obtain
\begin{equation*}
\begin{split}
\sum_{|x_1|,|x_2|\leq N}\left(\frac{x_1^2+\alpha_2x_2^2}{p}\right)
\ll&\frac{1}{p}\sum\limits_{\substack{h_1,h_2 \bmod p\\(h_1,h_2)\not\equiv (0,0) \bmod p}} \bigg|\sum_{|x_1|\leq N} e\left(-\frac{h_1x_1}{p}\right) \bigg|\cdot \bigg|\sum_{|x_2|\leq N}e\left(-\frac{h_2x_2}{p}\right)\bigg|\\
\ll&\frac{1}{p}\sum\limits_{\substack{h_1,h_2 \bmod p\\(h_1,h_2)\not\equiv(0,0) \bmod{p}}}\min\left\{N,\left|\left|\frac{h_1}{p}\right|\right|^{-1}\right\}\min\left\{N,\left|\left|\frac{h_2}{p}\right|\right|^{-1}\right\}\\
\ll&\frac{1}{p}\left(Np\log p +(p\log p)^2\right)\\
\ll_{\varepsilon}& (N+p)p^{\varepsilon},
\end{split}
\end{equation*}
$||z||$ denoting the distance of $z\in \mathbb{R}$ to the nearest integer.

Recalling \eqref{s1}, we now get 
$$
V_1\ll_{\varepsilon} q^{-1}\left(N^4+N^2p^2\right)p^{2\varepsilon}.
$$
For any $\Delta\in (0,1)$, the right-hand side is $O\left(\Delta N^6q^{-1}\right)$ if 
$$
N\ge p^{\varepsilon}\max\left\{\Delta^{-1/2},\Delta^{-1/4}p^{1/2}\right\}.
$$
Recalling $q\ge p^2$, it follows that 
\begin{equation} \label{condi3}
V_1=O_{\varepsilon}\left(\Delta N^6q^{-1}\right) \mbox{ if } N\ge q^{\varepsilon}\max\left\{\Delta^{-1/2},\Delta^{-1/4}q^{1/4}\right\}.
\end{equation}

\section{Estimation of $V_2$} 
To estimate $V_2$, defined in \eqref{s2}, we first note that 
\begin{equation} \label{V2ini}
V_2\le \frac{1}{\varphi(q)} \sum\limits_{\chi \bmod q}\bigg|\sum_{|x_1|,|x_2|\leq N} \chi\left(x_1^2+\alpha_2x_2^2\right) \bigg|^2 \cdot \max\limits_{\substack{\chi\bmod{q}\\ \chi\not=\chi_0}} \bigg| \sum_{|x_3|\leq N} \chi(x_3)\bigg|^2.
\end{equation}
Expanding the modulus square, and using orthogonality relations for Dirichlet characters, the sum over $\chi$ above transforms into 
\begin{equation} \label{double1}
\begin{split}
& \sum_{\chi\bmod q}\bigg| \sum_{|x_1|,|x_2|\leq N} \chi\left(x_1^2+\alpha_2 x_2^2\right)\bigg|^2\\
=&\sum_{\chi \bmod q}\ \sum\limits_{|x_1|,|x_2|,|y_1|,|y_2|\leq N} \chi\left(x_1^2+\alpha_2x_2^2\right)\overline{\chi}\left(y_1^2+\alpha_2y_2^2\right)\\
=&\varphi(q) \sum\limits_{\substack{|x_1|,|x_2|,|y_1|,|y_2|\leq N\\ (x_1^2+\alpha_2x_2^2,q)=1\\ (y_1^2+\alpha_2y_2^2,q)=1\\ x_1^2+\alpha_2x_2^2\equiv y_1^2+\alpha_2y_2^2 \bmod q}} 1.
\end{split}
\end{equation}
Furthermore, under the conditions $(\alpha_2,q)=1$ and $N<q/2$, we have 
\begin{equation} \label{double2}
\begin{split}
& \sum\limits_{\substack{|x_1|,|x_2|,|y_1|,|y_2|\leq N\\
x_1^2+\alpha_2x_2^2\equiv y_1^2+\alpha_2y_2^2\bmod q}}1\\
=& \sum\limits_{\substack{|x_1|,|x_2|,|y_1|,|y_2|\leq N \\ (x_1-y_1)(x_1+y_1)\equiv\alpha_2(y_2-x_2)(y_2+x_2) \bmod q}}1 \\
= &\sum\limits_{\substack{|x_1|,|x_2|,|y_1|,|y_2|\leq N\\x_1=\pm y_1\text{ and }x_2=\pm y_2}}1+\sum\limits_{\substack{0<|k_1|,|k_2|\leq 4N^2\\ k_1\equiv\alpha_2k_2 \bmod q}}\ \sum\limits_{\substack{|x_1|,|x_2|,|y_1|,|y_2|\leq N\\ (x_1-y_1)(x_1+y_1)=k_1\\(y_2-x_2)(y_2+x_2)=k_2}}1\\
\ll& N^2 + \sum_{0<|k_2|\leq 4N^2}\sum\limits_{\substack{0<|k_1|\leq 4N^2\\k_1\equiv \alpha_2k_2\bmod q}} \tau(|k_1|)\tau(|k_2|)\\
\ll_{\varepsilon} & N^{2+\varepsilon}\left(1+\frac{N^2}{q}\right),
\end{split}
\end{equation}
where we use the well-known bound $\tau(n)\ll_{\varepsilon} n^{\varepsilon}$ for the divisor function. From \eqref{double1} and \eqref{double2}, we obtain
\begin{equation} \label{double}
\sum_{\chi\bmod q}\bigg| \sum_{|x_1|,|x_2|\leq N} \chi\left(x_1^2+\alpha_2 x_2^2\right)\bigg|^2\\
\ll_{\varepsilon}qN^{2+\varepsilon}\left(1+\frac{N^2}{q}\right).
\end{equation}

We note that if $\chi$ is a non-principal character modulo $q$ which is induced by a primitive character $\chi_1$ modulo $q_1$ dividing $q$, then since $q$ is a prime power, 
$\chi(n)=\chi_1(n)$ for all $n\in \mathbb{Z}$. Hence, applying Propositions \ref{Burgess}, \ref{Mil} and \ref{Linde} with $q_1$ in place of $q$ and $r=2$ gives  
\begin{equation}\label{bb}
\max\limits_{\substack{\chi\bmod{q}\\ \chi\not=\chi_0}} \bigg|\sum_{|x_3|\leq N}\overline{\chi}(x_3)\bigg|^2 =\begin{cases} O_{\varepsilon}\left(Nq^{3/8+\varepsilon}\right)\\ \\
O_{\varepsilon,p} \left(N^{2(l-k)}q^{2k+\varepsilon}\right)\\ \\
O_{\varepsilon} \left(Nq^{\varepsilon}\right) \mbox{ under the Lindel\"of hypothesis.}
\end{cases}
\end{equation}
Combining \eqref{V2ini}, \eqref{double} and \eqref{bb}, we find that 
$$
V_2= \begin{cases} O_{\varepsilon}\left(\left(1+N^2q^{-1}\right)N^3q^{3/8+\varepsilon}\right)\\ \\
O_{\varepsilon,p} \left(\left(1+N^2q^{-1}\right)N^{2(1+l-k)}q^{2k+\varepsilon}\right)\\ \\
O_{\varepsilon} \left(\left(1+N^2q^{-1}\right)N^3q^{\varepsilon}\right) \mbox{ under the Lindel\"of hypothesis.}
\end{cases}
$$   

We observe that for $u,v\in \mathbb{R}^2$ with $u<6$ and $\Delta\in (0,1)$,
$$
N^uq^v\le \Delta N^6q^{-1}\Longleftrightarrow N\ge \Delta^{-1/(6-u)} q^{(v+1)/(6-u)}.
$$
It follows that 
\begin{equation} \label{condi4}
V_2=\begin{cases} O_{\varepsilon}\left(\Delta N^6q^{-1}\right) \mbox{ if } N\ge q^{\varepsilon}\max\left\{\Delta^{-1/3}q^{11/24},\Delta^{-1}q^{3/8}\right\}\\ \\
O_{\varepsilon,p}\left(\Delta N^6q^{-1}\right) \mbox{ if } N\ge q^{\varepsilon}\mu_{k,l}(\Delta,q) \\ \\
O_{\varepsilon} \left(\Delta N^6q^{-1}\right) \mbox{ if } N\ge q^{\varepsilon}\max\left\{\Delta^{-1/3}q^{1/3},\Delta^{-1}\right\}\mbox{ under the Lindel\"of hypothesis.}
\end{cases}
\end{equation}
with 
\begin{equation} \label{mukldef}
\mu_{k,l}(\Delta,q):=\max\left\{\Delta^{-1/(4+2k-2l)}q^{(2k+1)/(4+2k-2l)},\Delta^{-1/(2+2k-2l)}q^{2k/(2+2k-2l)}\right\}.
\end{equation}
Taking the exponent pair $(k,l)=ABA^2B(0,1)=(1/9,13/18)$, the above equals
\begin{equation} \label{concretekl}
\mu_{k,l}(\Delta,q)=\max\left\{\Delta^{-9/25}q^{11/25},\Delta^{-9/7}q^{2/7}\right\}.
\end{equation}

\section{Proof of Theorem \ref{mainresult}}
Let $\Delta\in (0,1)$, $N\in \mathbb{N}$, and assume that $2N\le q^{7/12}$. Then combining \eqref{Vsplit}, \eqref{condi2}, \eqref{condi3}, \eqref{condi4} and \eqref{concretekl}, we have
\begin{equation} \label{condi5}
V=\begin{cases} O_{\varepsilon}\left(\Delta N^6q^{-1}\right) \mbox{ if } N\ge \Delta^{-3}q^{11/24+\varepsilon}\\ \\
O_{\varepsilon,p}\left(\Delta N^6q^{-1}\right) \mbox{ if } N\ge \Delta^{-3}q^{11/25+\varepsilon}\\\ \\
O_{\varepsilon} \left(\Delta N^6q^{-1}\right) \mbox{ if } N\ge \Delta^{-3}q^{1/3+\varepsilon}\mbox{ under the Lindel\"of hypothesis.}
\end{cases}
\end{equation}
(Here we bound all powers of $\Delta$ occurring in \eqref{condi2}, \eqref{condi3}, \eqref{condi4} and \eqref{concretekl} by $\Delta^{-3}$.)
Recalling our choice $\Delta:=q^{-2\varepsilon}$ and assuming that $\varepsilon$ is small enough, it follows from \eqref{Mapproxi}, \eqref{Vdefi} and \eqref{condi5} that
\begin{equation} \label{condi6}
\sum\limits_{\substack{\alpha_3=1\\ (\alpha_3,q)=1}}^q \left|S(\alpha_3)-C_q\cdot \frac{N^3}{q}\right|^2=\begin{cases} O_{\varepsilon}\left(N^6q^{-2\varepsilon-1}\right) \mbox{ if } N\ge q^{11/24+7\varepsilon}\\ \\
O_{\varepsilon,p}\left(N^6q^{-2\varepsilon-1}\right) \mbox{ if } N\ge q^{11/25+7\varepsilon}\\\ \\
O_{\varepsilon} \left(N^6q^{-2\varepsilon-1}\right) \mbox{ if } N\ge q^{1/3+7\varepsilon}\mbox{ under the Lindel\"of hypothesis.}
\end{cases}
\end{equation}
Now we observe that if the left-hand side of \eqref{condi6} is $O\left(N^6q^{-2\varepsilon-1}\right)$, then we have 
$$
S(\alpha_3)=C_q\cdot \frac{N^3}{q}\cdot \left(1+O\left(q^{-\varepsilon/2}\right)\right)
$$
for all 
$$
\alpha_3\in \{s\in \mathbb{N} : 1\le s\le q, \ (s,q)=1\}
$$ 
with at most $O\left(\varphi(q)q^{-\varepsilon}\right)$ exceptions. This together with \eqref{condi6} implies the result of Theorem \ref{mainresult} upon redefining $\varepsilon$. \\ \\
{\bf Remark:} From \eqref{condi4} and \eqref{mukldef} it is apparent that in the situation of large prime powers of a fixed prime $p$, the task becomes to minimize the fraction
$$
f(k,l)=\frac{2k+1}{4+2k-2l}
$$
over the set of exponent pairs $(k,l)$. This can be done using an algorithm by Graham \cite{Gr} which puts out a sequence of transformations consisting of $A$ or $BA$ which determines exponent pairs approximating the optimal one. Following this algorithm, we calculated that the minimal value of $f(k,l)$ lies in the interval 
$$
(0.439875556384,0.439875557961),
$$  
where the lower bound arises from applying the process $ABABABA^3BA^2BABABA^2BABABA$ to $(0,1/2)$ and the upper bound from applying this process to $(0,1)$. 
In our paper, we stop when reaching the exponent pair $(k,l)=(1/9,13/18)=ABABABA(0,1)=ABA^2B(0,1)$. This gives an exponent of $f(k,l)=11/25=0.44$, which is obviously smaller than the exponent $11/24=0.458\overline{3}$ in part (i) of Theorem \ref{mainresult}.

\end{document}